\documentclass[12pt,draftclsnofoot,onecolumn]{IEEEtran}

\usepackage{amsmath,amssymb}
\usepackage{pdflscape}

\usepackage{graphicx,calc,cite,bm,amsthm,enumerate,cases,mathtools,accents,cuted}
\renewcommand{\(}{\left(}
\renewcommand{\)}{\right)}

\renewcommand{\]}{\right]}

\newcommand{\N}{\mathbb{N}}

\newcommand{\End}[1]{{\rm{End}}}

\renewcommand{\log}[1]{{\rm{log}}#1}

\newtheorem{lemma}{Lemma}
\newtheorem{definition}{Definition}
\newtheorem{theorem}{Theorem}

\newtheorem{corollary}{Corollary}

\newtheorem{rem}{Remark}

\newlength{\depthofsumsign}
\newcommand{\nsum}[1][1.4]{
    \mathop{%
        \raisebox
            {-#1\depthofsumsign+1\depthofsumsign}
            {\scalebox
                {#1}
                {$\displaystyle\sum$}%
            }
    }
}

\newcommand{\nprod}[1][1.4]{
    \mathop{%
        \raisebox
            {-#1\depthofsumsign+1\depthofsumsign}
            {\scalebox
                {#1}
                {$\displaystyle\prod$}%
            }
    }
}

\usepackage[ruled,vlined,resetcount]{algorithm2e}
\usepackage{algorithmic,subfigure}
\usepackage{fontenc,amsfonts}
\usepackage{inputenc,relsize}
\usepackage[square,sort,compress,comma,numbers]{natbib}
\usepackage{epstopdf,pst-node}
\usepackage{siunitx}

\usepackage{ytableau,tikz,varwidth}
\usetikzlibrary{calc}

\usepackage{ytableau}

\usepackage{pgfplots,authblk}
\usepackage{verbatim}
\usepackage{smartdiagram}
\usesmartdiagramlibrary{additions}

\usepackage{pgf}

\SetCommentSty{mycommfont}

\begin{document}
\title{Large Deviations of Convex Polyominoes}

\author[1]{Ilya Soloveychik\thanks{This work was supported by the Fulbright Foundation and Office of Navy Research grant N00014-17-1-2075.}}
\author[2]{Vahid Tarokh}
\affil[1]{\normalsize John A. Paulson School of Engineering and Applied Sciences, Harvard University}
\affil[2]{\normalsize Department of Electrical and Computer Engineering, Duke University}
\maketitle

\begin{abstract}
Enumeration of various types of lattice polygons and in particular polyominoes is of primary importance in many machine learning, pattern recognition, and geometric analysis problems. In this work, we develop a large deviation principle for convex polyominoes under different restrictions, such as fixed area and/or perimeter. 
\end{abstract}


\section{Introduction}
The paramount interest to the theoretical and practical aspects of machine learning and pattern recognition over the last ten years has led to many important breakthroughs in the underlying algorithms and techniques. One of the most prominent tools allowing to asses the quality of these algorithms consists in derivation of lower theoretic bounds serving as benchmarks in comparative studies. Such bounds quantify the best achievable performance (e.g. Cramer-Rao bounds), sample, computational, or algorithmic complexity in the problem at hand, etc. The most critical component of many of such results is the estimation of the number of admissible models one of which is to be chosen as the outcome of the learning process. For example, in pattern recognition the goal is to select a pattern from a family of models best suiting the input data under some performance criteria. The sample complexity of such model selection is controlled by the richness of the set of eligible models, see e.g. \cite{santhanam2012information, anandkumar2012high, soloveychik2018region}. In its turn, the problem of counting the cardinalities of model classes in this scenario is equivalent to the problem of enumeration of certain geometric shapes.

Enumeration of geometric shapes has been one of the central combinatorial problems for a long time \cite{vershik2004limit, guttmann2009polygons}. Recently discovered connections to the theory of random partitions and concentration of measure have drawn attention of many scientists to this area. Let us start by mentioning the fundamental works of Vershik, Blinovskii, Dembo and Zeitouni \cite{vershik1987statistical, blinovskii1999large, dembo1998large}, who developed large deviation principle for integer partitions. In particular they showed that the boundaries of the $\frac{1}{\sqrt{n}}$-scaled Young diagrams corresponding to the partitions of the integer $n$ endowed with the uniform measure concentrate around a non-random limiting curve. In \cite{vershik1987statistical} Vershik calculated the exact shape of the curve. Given an arbitrary curve satisfying some natural regularity conditions, \cite{blinovskii1999large, dembo1998large} derived the exact speed and rate function controlling the number of scaled Young diagrams in a small vicinity of the curve. This line of research was further extended by other mathematicians to different setups and conditions. In \cite{dembo1998large} a large deviation principle for strict partitions was derived, in \cite{vershik1994limit, vershik1999large} -  for convex polygons on an integer lattice, etc. In some cases only the limiting curve was obtained without the large deviation principle, e.g. the case of restricted and boxed partitions \cite{petrov2009two}. 

Consider the integer lattice on the $\mathbb{R}^2$ plane. A lattice \textit{polyomino} is a union of elementary lattice cells which must be joined at their sides \cite{guttmann2009polygons}. A polyomino is said to be \textit{column-convex} in a given lattice direction if all the cells along any line in that direction are connected through cells in the same line. A polyomino on the integer lattice is \textit{convex} if it is column-convex in both horizontal and vertical directions. One of the main problems in the field of convex polyominoes is their enumeration \cite{guttmann2009polygons}. There exists a large body of literature addressing the problem of polyomino counting according to their perimeter and/or area \cite{delest1988generating, bousquet1992convex, bousquet1996method, bousquet1995generating}. However, in all these works the desired numbers are given implicitly as coefficients of the corresponding terms in the series expansions of the generating functions derived therein. These series are usually too complicated and bulky to be directly analyzed and the sought for coefficients cannot be easily extracted. Moreover, even the asymptotic behavior of these coefficients is by no means obvious to derive. Using the ideas from \cite{blinovskii1999large}, in this work we develop a large deviation principle for convex polyominoes with different constraints, such as perimeter, area, or both. To the best of our knowledge the English version of \cite{blinovskii1999large} published in Problems of Information Transmission in 1999 is not in open access, therefore, for completeness we repeat the main arguments from this seminal paper in our proofs in Section \ref{sec:proofs}. Interestingly, our findings generalize some of the results in the works devoted to the study of equilibrium shapes of convex polyominoes of fixed perimeter under different pressure \cite{mitra2008asymptotic, mitra2010asymptotic}.

The rest of the text is organized as follows. First we introduce the large deviation principle and the necessary notation in Section \ref{sec:ldp_def}. In Section \ref{sec:conv_polyominoes} we define convex polyominoes and discuss their geometric properties. We formulate the main results in Section \ref{sec:main_res}. In Section \ref{seq:appl} we discuss an example of application of the obtained results. Finally, Section \ref{sec:proofs} contains the proofs.

\section{The Large Deviation Principle}
\label{sec:ldp_def}
In this section, we introduce the notion of Large Deviation Principle (LDP). Our main result concerning the enumeration of convex polyominoes will be formulated in terms of LDP. Let $\mathcal{P}$ be a Polish space (complete separable metric space). Given $B \subset \mathcal{P}$, denote by $B^0$ the interior of $B$ and by $\bar{B}$ its closure.
\begin{definition}
A sequence $\{\mathbb{P}_n\}_{n=1}^\infty$ of probability measures on $\mathcal{P}$ satisfies a Large Deviation Principle with speed $a_n$ and rate function $I$ if
\begin{equation}
-\inf_{b \in B^0}I(b) \leqslant \liminf_{n \to \infty} \frac{\log\, \mathbb{P}_n(B)}{a_n} \leqslant \limsup_{n \to \infty} \frac{\log\,\mathbb{P}_n(B)}{a_n} \leqslant -\inf_{b \in \bar{B}}I(b),\quad \forall B \subset \mathcal{P},
\end{equation}
where $I : \mathcal{P} \to \overline{\mathbb{R}}_+$ is lower semi-continuous (its level sets $L(M) = \{b\in \mathcal{P} | I(b)\leqslant M\}$ are closed for any $M\geqslant 0$). If $L(M)$ are compact, we refer to $I$ as a good rate function.
\end{definition}

Given an element $\gamma \in \mathcal{P}$, let $U_\varepsilon(\gamma)$ be its $\varepsilon$-vicinity. In addition to the LDP we also formulate the so-called local LDP.
\begin{definition}
Assume that for all $\gamma \in \mathcal{P}$,
\begin{equation}
\liminf_{\varepsilon \to 0}\liminf_{n \to \infty} \frac{\log\, \mathbb{P}_n(U_\varepsilon(\gamma))}{a_n} =\limsup_{\varepsilon \to 0}\limsup_{n \to \infty} \frac{\log\, \mathbb{P}_n(U_\varepsilon(\gamma))}{a_n} = -I(\gamma),
\end{equation}
then we say that $\mathbb{P}_n$ satisfies the local LDP.
\end{definition}

The last definition can be roughly interpreted as
\begin{equation}
\mathbb{P}_n(U_\varepsilon(\gamma)) \sim e^{-a_nI(\gamma)}.
\end{equation}

\section{Convex Polyominoes}
\label{sec:conv_polyominoes}
\begin{figure}[!t]
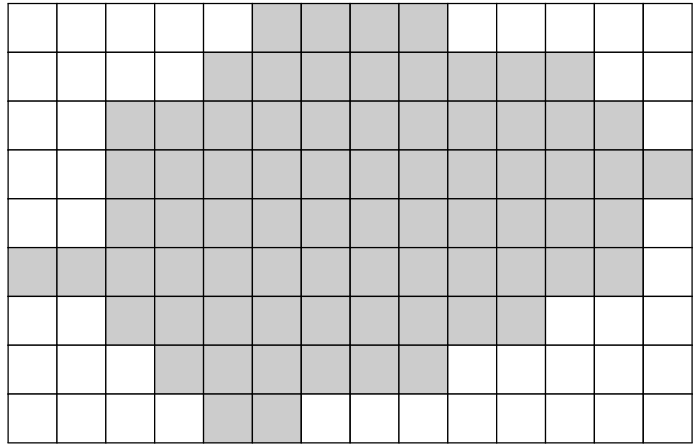

\centering
 \begin{equation*}
 \ytableausetup{textmode}
\begin{ytableau}
*(black!0) & *(black!0) & *(black!0) & *(black!0) & *(black!0) & *(black!20) & *(black!20) & *(black!20) & *(black!20) & *(black!0) & *(black!0) & *(black!0) & *(black!0) & *(black!0) \\
*(black!0) & *(black!0) & *(black!0) & *(black!0) & *(black!20) & *(black!20) & *(black!20) & *(black!20) & *(black!20) & *(black!20) & *(black!20) & *(black!20) & *(black!0) & *(black!0) \\
*(black!0) & *(black!0) & *(black!20) & *(black!20) & *(black!20) & *(black!20) & *(black!20) & *(black!20) & *(black!20) & *(black!20) & *(black!20) & *(black!20) & *(black!20) & *(black!0) \\
*(black!0) & *(black!0) & *(black!20) & *(black!20) & *(black!20) & *(black!20) & *(black!20) & *(black!20) & *(black!20) & *(black!20) & *(black!20) & *(black!20) & *(black!20) & *(black!20) \\
*(black!0) & *(black!0) & *(black!20) & *(black!20) & *(black!20) & *(black!20) & *(black!20) & *(black!20) & *(black!20) & *(black!20) & *(black!20) & *(black!20) & *(black!20) & *(black!0) \\
*(black!20) & *(black!20) & *(black!20) & *(black!20) & *(black!20) & *(black!20) & *(black!20) & *(black!20) & *(black!20) & *(black!20) & *(black!20) & *(black!20) & *(black!20) & *(black!0) \\
*(black!0) & *(black!0) & *(black!20) & *(black!20) & *(black!20) & *(black!20) & *(black!20) & *(black!20) & *(black!20) & *(black!20) & *(black!20) & *(black!0) & *(black!0) & *(black!0) \\
*(black!0) & *(black!0) & *(black!0) & *(black!20) & *(black!20) & *(black!20) & *(black!20) & *(black!20) & *(black!20) & *(black!0) & *(black!0) & *(black!0) & *(black!0) & *(black!0) \\
*(black!0) & *(black!0) & *(black!0) & *(black!0) & *(black!20) & *(black!20) & *(black!0) & *(black!0) & *(black!0) & *(black!0) & *(black!0) & *(black!0) & *(black!0) & *(black!0) \\
\end{ytableau}
\end{equation*}
\caption{\footnotesize A convex polyomino.}
\label{fig:conv_polyomino}
\end{figure}

Given the integer lattice on $\mathbb{R}^2$, a lattice \textit{polyomino} is a union of elementary lattice cells which must be joined at their sides and not just at nodes \cite{guttmann2009polygons}, such as e.g. the cells colored gray in Figure \ref{fig:conv_polyomino}. A polyomino is said to be \textit{column-convex} in a given lattice direction if all the cells along any line in that direction are connected through cells in the same line. A polyomino on the integer lattice is \textit{convex} if it is column-convex in both horizontal and vertical directions.
\begin{lemma}[\cite{guttmann2009polygons}]
\label{lem:conv_circ}
A square lattice polyomino is convex if and only if its perimeter coincides with the perimeter of its circumscribed rectangle.
\end{lemma}
Figure \ref{fig:conv_polyomino} shows an example of a convex polyomino on a square lattice and its circumscribed rectangle. In the discrete scenario we have the following analog of the isoperimetric inequality.

\begin{lemma}[Isoperimeteric inequality on the square lattice]
\label{lem:isoper_ineq}
For a polyomino of area $A$ and perimeter $L$ on the square lattice,
\begin{equation}
\label{eq:isoper_ineq}
A \leqslant \frac{L^2}{16},
\end{equation}
the equality is reached when the polyomino is a square.
\end{lemma}
\begin{proof}
We must only prove (\ref{eq:isoper_ineq}) for convex polyominoes. Due to Lemma \ref{lem:conv_circ}, the perimeter of the circumscribed rectangle of a convex polyomino of perimeter $L$ is also $L$. Clearly, the area of such a polyomino is maximized when it coincides with its circumscribed rectangle. Among the rectangles of perimeter $L$, the area is maximal for the square, which completes the proof.
\end{proof}

\section{Large Deviation Principle for Convex Polyominoes} 
\label{sec:main_res}
Consider the plane $\mathbb{R}^2$ with the standard basis and fixed origin. Assume that we are given a closed piece-wise differentiable curve $\Gamma \subset \mathbb{R}^2$ which is \textit{unimodal} in both vertical and horizontal directions. In other words, every horizontal and vertical line intersects the curve in at most two points. Denote the region embraced by $\Gamma$ by $G$ and its area by
\begin{equation}
\text{area}(G) = A.
\end{equation}
For convenience, let us assume that the barycenter of $G$ coincides with the coordinate origin.\footnote{A generalization to the case where the boundary of $\Gamma$ may contain vertical and horizontal segments in such a way that any horizontal and vertical line would cross $G$ along one contiguous segment is straightforward. We stick to the unimodal case to make the proofs in Section \ref{sec:proofs} less technically involved.} Given two curves $\Gamma_1$ and $\Gamma_2$, the distance between them is defined as
\begin{equation}
d(\Gamma_1,\Gamma_2) = \text{area}(G_1\Delta G_2).
\end{equation}

For every $n \in \mathbb{N}$ we construct the integer lattice centered at the origin and scale it by $\frac{1}{\sqrt{n}}$ so that the area of every elementary cell becomes $\frac{1}{n}$. Consider the set of convex polyominoes in the $\varepsilon$-vicinity of $\Gamma$, which we denote by
\begin{equation}
\mathbb{Q}_n = \mathbb{M}_n \cap U_{\varepsilon}(\Gamma),
\end{equation}
where $\mathbb{M}_n$ is the set of all convex polyominoes on the $\frac{1}{\sqrt{n}}$-grid.

Our goal will be to count the polyominoes in $\mathbb{Q}_n$ satisfying different conditions. For example, the polyominoes in $\mathbb{Q}_n$ having fixed area $\mathbb{Q}_{A}$\footnote{We suppress the $n$ index to simply the notation.}, fixed perimeter $\mathbb{Q}_{L}$, or both fixed area and perimeter $\mathbb{Q}_{A,L}$, etc. Denote
\begin{equation}
Q_X = |\mathbb{Q}_{X}|,\quad X = A,L,\{A,L\}.
\end{equation}


\begin{rem}
By convention, below we write
\begin{equation}
\int_{\Gamma}f(\Gamma)(|dx|+|dy|) = \int_{\Gamma}f(\Gamma(s))(|\sin(\theta)|+|\cos(\theta)|)ds,
\end{equation}
where $\Gamma(s)$ is the natural parameterization of the curve by its arc length and $\theta = \arctan(y')$ is the angle between the tangent line at any point and the horizontal axis.
\end{rem}

On the sets of the form $U_\varepsilon(\Gamma)$ define our measures $\mathbb{P}_{X,n}$ as
\begin{equation}
\mathbb{P}_{X,n}(U_\varepsilon(\Gamma)) = \frac{Q_X(U_\varepsilon(\Gamma))}{V_X}, \quad X = A,L,\{A,L\}.
\end{equation}
where $V_X$ is the total number of convex polyominoes of type $X$.

For concreteness, let us first consider convex polyominoes of fixed area, $X=A$ and denote $\mathbb{P}_{n} = \mathbb{P}_{A,n}$. Below we explain that the other cases are treated analogously.

\begin{figure}[!t]
\centering
\includegraphics[width=3.2in]{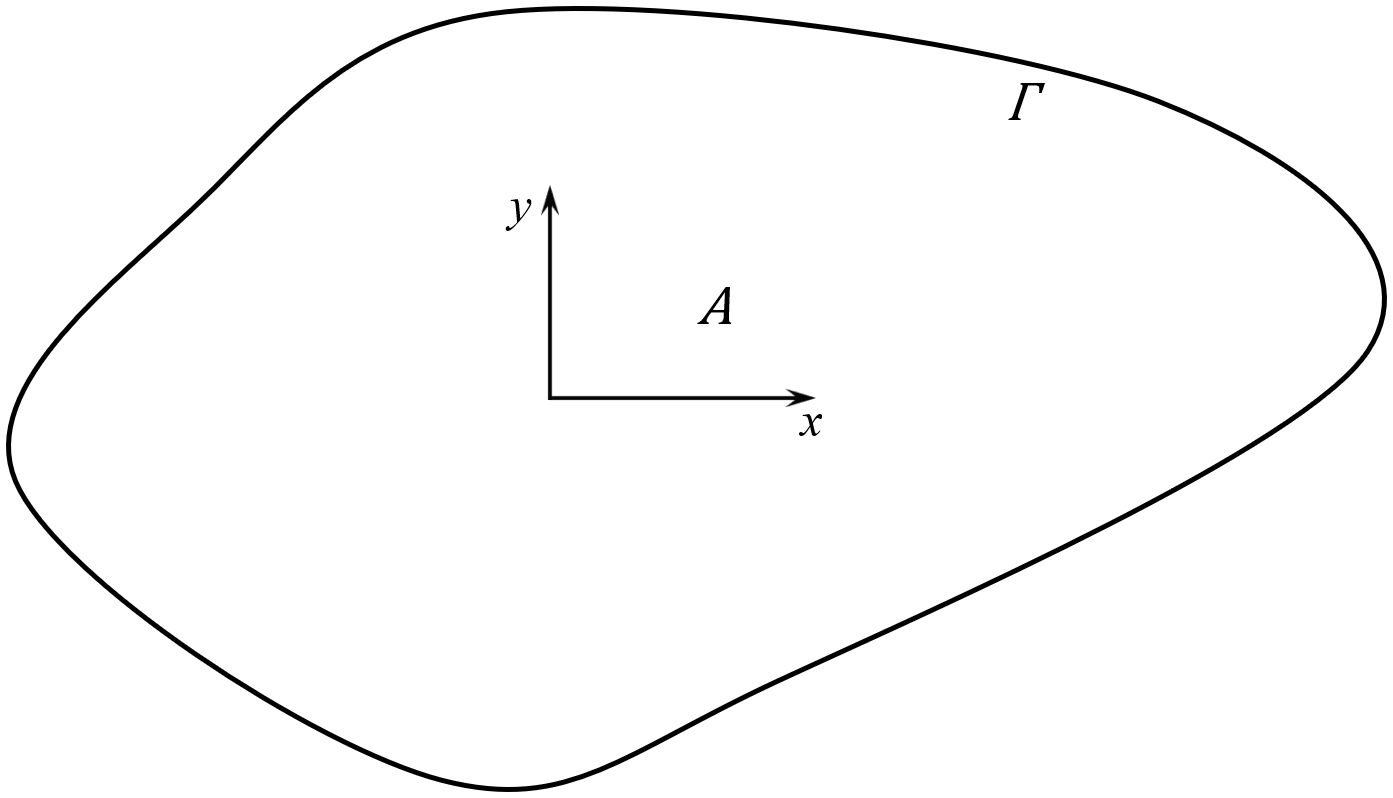}
\vspace{-0.2cm}
\caption{\footnotesize The original unimodal curve $\Gamma \subset \mathbb{R}^2$.}
\label{fig:dens_13}
\end{figure}

\begin{theorem}[LDP for Convex Polyominoes]
\label{thm:ldp_main}
Let $\Gamma$ be a unimodal in the vertical and horizontal directions piece-wise differentiable curve embracing a region of area $A$. Then the sequence $\{\mathbb{P}_{n}\}_n$ satisfies the local LDP with speed $\sqrt{n}$ and good rate function
\begin{equation}
I(\Gamma) = C_A - \int_{\Gamma}H\(\frac{|y'|}{1+|y'|}\)(|dx|+|dy|) = C_A - \int_{\Gamma}(1+|\tan \Gamma|)H\(\frac{1}{1+|\cot \Gamma|}\)|dx|,
\end{equation}
where $H(u)=-u\,\log_2 u - (1-u)\log_2 (1-u)$ is the binary entropy\footnote{Below we suppress the subscript of the logarithm for brevity.}, $y=y(x)$ is the local parametrization of the curve, and $C_A$ is the normalization constant (log-partition function in the statistical mechanics terminology).
\end{theorem}

As an immediate corollary, we obtain the following statement, where we count the actual number of the polyominoes and not the (normalized) probability. This allows us to get rid of the constant $C_A$.
\begin{corollary}
\label{cor:main_cor}
The number of convex polyominoes of area $A$ inside $U_\varepsilon(\Gamma)$ satisfies
\begin{align}
\liminf_{\varepsilon \to 0}&\liminf_{n \to \infty} \frac{\log\, Q_A\(U_\varepsilon(\Gamma)\)}{\sqrt{n}} = \limsup_{\varepsilon \to 0}\limsup_{n \to \infty} \frac{\log\, Q_A\(U_\varepsilon(\Gamma)\)}{\sqrt{n}} \nonumber \\
& =  \int_{\Gamma}H\(\frac{|y'|}{1+|y'|}\)(|dx|+|dy|) = \int_{\Gamma}(1+|\tan \Gamma|)H\(\frac{1}{1+|\cot \Gamma|}\)|dx|.
\end{align}
\end{corollary}

\begin{rem}
Similar results can be obtained for convex polyominoes with fixed perimeter and with both fixed area and perimeter. The only difference will be in the value of the constant $C_A$. For  of type $X$, this constant can be calculated as
\begin{equation}
\label{eq:curv_max}
C_X = \max_{\Gamma \in X} \int_{\Gamma}(1+|\tan \Gamma|)H\(\frac{1}{1+|\cot \Gamma|}\)|dx|. 
\end{equation}
It is easy to see that the curves on which the extremum is reached \cite{petrov2009two} are concatenations of the properly scaled segments of Vershik's limiting shape \cite{vershik1987statistical} given by the equation
\begin{equation}
\label{eq:versh_curve}
e^{\frac{-\pi x}{\sqrt{6}}}+e^{\frac{-\pi y}{\sqrt{6}}}=1.
\end{equation}
In order to find the segments of this curve that maximize (\ref{eq:curv_max}) for the family $X$ under consideration, we need to find such parts of Vershik's curve (\ref{eq:versh_curve}) that satisfy the required relations between the perimeter (coinciding with the perimeter of the circumscribed rectangle) and area. As shown in \cite{petrov2009two} for any admissible combination of $L$ and $A$ we can always find the necessary segments on the curve (\ref{eq:versh_curve}).
\end{rem}

\begin{figure}[!t]
\centering
\includegraphics[width=3.2in]{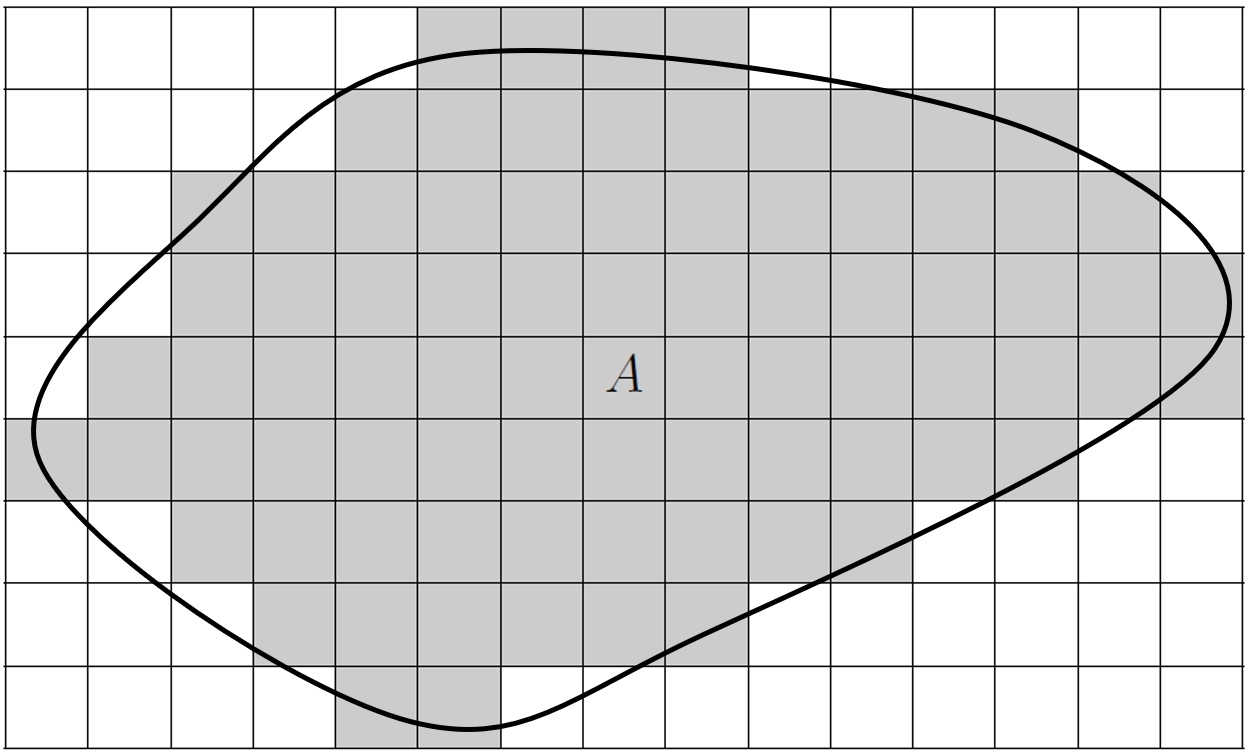}
\vspace{-0.2cm}
\caption{\footnotesize The curve and the approximating convex polyomino.}
\label{fig:dens_13}
\end{figure}

\begin{figure}[!t]
\centering
\includegraphics[width=3.2in]{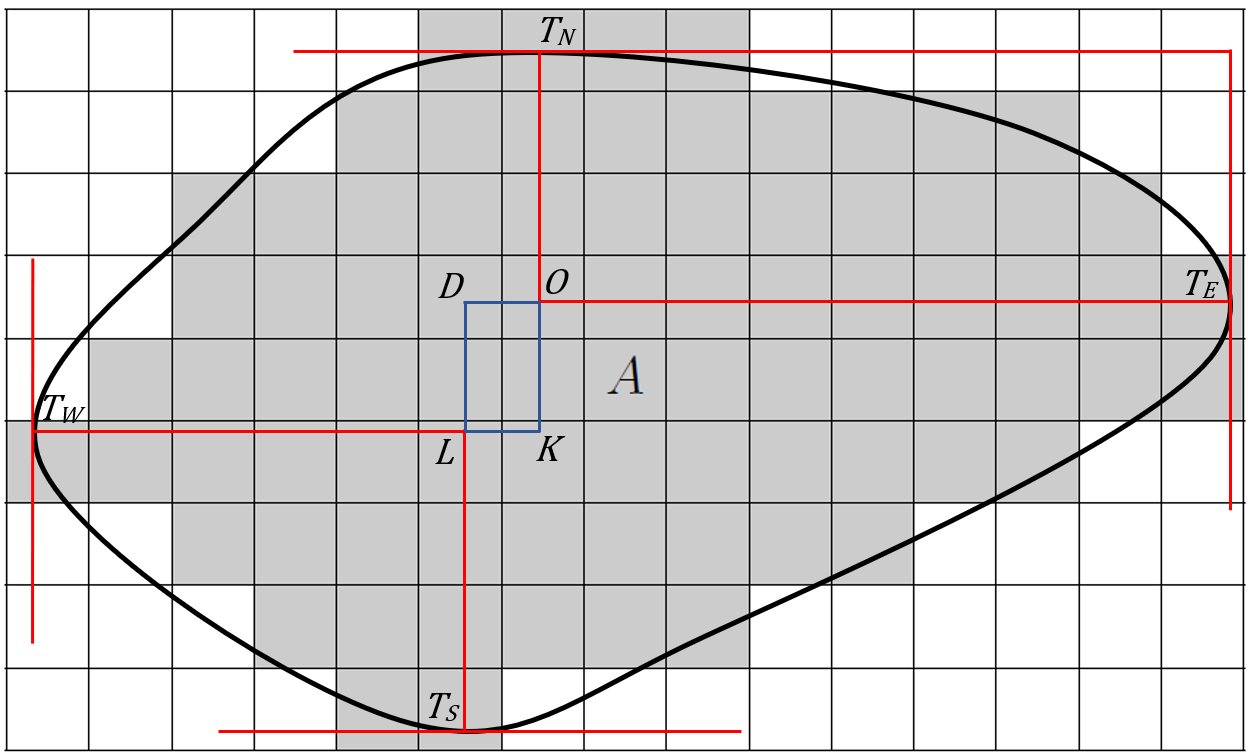}
\vspace{-0.2cm}
\caption{\footnotesize Partitioning of the curve.}
\label{fig:part_gamma}
\end{figure}


\section{Applications}
\label{seq:appl}
A remarkable application of the results stated in Theorem \ref{thm:ldp_main} and Corollary \ref{cor:main_cor} is provided by a recent work \cite{soloveychik2018region} which motivated the present study. That paper focuses on the problem of graphical model selection in sample starving regime when the number of measurements is not sufficient for the consistent detection of all the edges of the graph. Under such restrictions, the authors of \cite{soloveychik2018region} reconsider the goal of the learning process and instead of discovering all the edges of the graph formulate the problem as detection of graph regions with similar coupling parameters and regular boundaries. The regions are modeled by convex polyominoes on square lattices. In order to develop lower information-theoretic bounds derived through the application of Fano's inequality, the authors of \cite{soloveychik2018region} estimate the logarithm of the number of admissible models. The result of the present article allowed them to obtain such bounds and demonstrate that the region detection is feasible with a bounded number of samples unlike the classical graphical model selection requiring this number to grow at least as the logarithm of the number of graph vertices.

\section{Proofs}
\label{sec:proofs}
This section is devoted to the proof of the main result and contains a number of auxiliary lemmas.
\begin{proof}[Proof of Theorem \ref{thm:ldp_main}]
According to the definition of the LDP, the proof will be complete if we demonstrate that
\begin{equation}
\label{eq:proof_1m}
\limsup_{\varepsilon \to 0}\limsup_{n \to \infty} \frac{\log\, \mathbb{P}_n(U_\varepsilon(\Gamma))}{\sqrt{n}} \leqslant -I(\Gamma),
\end{equation}
and
\begin{equation}
\label{eq:proof_2m}
\liminf_{\varepsilon \to 0}\liminf_{n \to \infty} \frac{\log\, \mathbb{P}_n(U_\varepsilon(\Gamma))}{\sqrt{n}} \geqslant -I(\Gamma).
\end{equation}

Let us start with (\ref{eq:proof_1m}). By the very definition of $\Gamma$, it can be partitioned into four segments each of which is a graph of a strictly monotonic function. In our example in Figure \ref{fig:part_gamma}, the four segments are the curve arcs $T_NT_E,\; T_ET_S,\; T_ST_W$, and $T_WT_N$ connecting the points of intersection of $\Gamma$ with its two horizontal and two vertical tangent lines. 

Given a polyomino $Z \in \mathbb{Q}_A$, consider its top row of cells and choose the center of one of these cells. We call the obtained point the north extreme point of the polyomino and denote it by $N$ (see Figure \ref{fig:zoom_extr}). Analogously, we define the other extreme points $E,\; S$, and $W$. 

Now let us consider all the polyominoes from $\mathbb{Q}_A$ whose $N$ and $S$ extreme points have the same $x$ coordinate and whose $E$ and $W$ extreme points have the same $y$ coordinate. Denote this set by $\mathbb{Q}_A^e$ and let us bound its cardinality from above. Indeed, 
\begin{equation}
\label{eq:z_e_calc}
Q_A^e \leqslant Q_A^{T_NN}Q_A^{NT_E}Q_A^{T_EE}Q_A^{ES}Q_A^{ST_S}Q_A^{T_ST_W}Q_A^{T_WW}Q_A^{WT_N},
\end{equation}
where $Q_A^{T_NN}$ is the number of decreasing diagrams in the $\varepsilon$-vicinity of $T_NR$, where $R$ is the point of intersection of the vertical line through $N$ with $\Gamma$ that fit in between the vertical lines through $T_N$ and $N$, $Q_A^{NT_E}$ is the number of decreasing diagrams in the $\varepsilon$-vicinity of $RT_E$ belonging to the quadrant to the north-east from the vertical line through $N$ and the horizontal line through $T_E$, and so on in an analogous manner. For convenience, take the logarithm of both sides of (\ref{eq:z_e_calc}) to obtain
\begin{align}
\label{eq:z_e_calc_log}
\log\,Q_A^e & \leqslant \log\,Q_A^{NT_E} + \log\,Q_A^{T_ST_W} + \log\,Q_A^{ES} + \log\,Q_A^{WT_N} \nonumber \\
& + \log\, Q_A^{T_NN} + \log\,Q_A^{T_EE} + \log\,Q_A^{ST_S} + \log\,Q_A^{T_WW}.
\end{align}
Our goal will be to show that the main contribution to (\ref{eq:z_e_calc_log}) is made by the diagrams inside the large quadrants (the first row in the righ-hand side of (\ref{eq:z_e_calc_log})) whereas those parts of the boundary that correspond to the segments of the form $Q_A^{XT_X}$ or $Q_A^{T_XX}$ (the second row) tend to zero as $\varepsilon$ approaches zero. 

\begin{figure}[!t]
\centering
\includegraphics[width=3.38in]{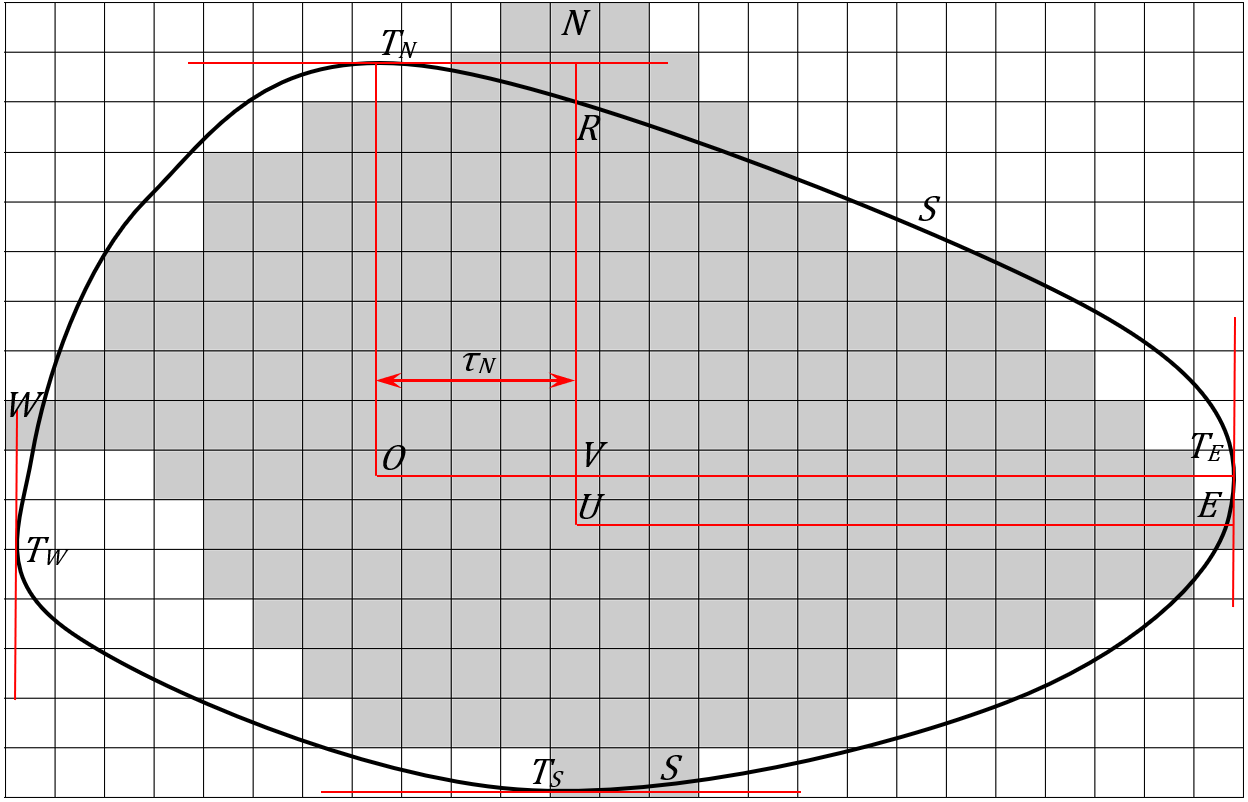}
\vspace{-0.2cm}
\caption{\footnotesize Deviation of the polyomino from the middle curve.}
\label{fig:zoom_extr}
\end{figure}

Let us start from bounding the value of $\log\,Q_A^{T_NN}$. Indeed, the horizontal distance between the points $T_N$ and $N$ must shrink with $\varepsilon$ because the curve is strictly monotonic,
\begin{equation}
\tau_N \to 0,\quad \varepsilon \to 0.
\end{equation}
Analogous relations hold for the other $\tau_X$ as well,
\begin{equation}
\label{eq:tau_bound}
\tau_X \to 0,\;\; \varepsilon \to 0, \quad X = N,E,S,W.
\end{equation}

Below we use the following basic result.
\begin{lemma}(Diagrams With Fixed Endpoints)
\label{lem:count_diagr}
The number of monotonic diagrams connecting points $(a_1,b_1)$ and $(a_2,b_2)$ of the integer square lattice with $a_1\neq a_2$ and $b_1\neq b_2$ which are also right-continuous at $(a_1,b_1)$ is given by
\begin{equation}
\label{eq:binom_coef_enum}
N((a_1,b_1),(a_2,b_2)) = {|a_1-a_2|+|b_1-b_2|-1 \choose |b_1-b_2|}.
\end{equation}
\end{lemma}

\begin{proof}
Let us associate $1$ with every horizontal edge and $0$ with every vertical edge. Then the desired number of diagrams coincides with the amount of ways $|a_1-a_2|-1$ ones and $|b_1-b_2|$ zeros can be written into a binary codeword of length $|a_1-a_2|+|b_1-b_2|-1$, which is given by the binomial coefficient in (\ref{eq:binom_coef_enum}).
\end{proof}

\begin{lemma}[Binomial Coefficient Bound]
Let $a,\; b\in \N$ and $a > b$, then
\begin{equation}
\label{eq:stril_b}
aH\(\frac{b}{a}\) - \log \(\sqrt{8\pi b(1-b/a)}\) \leqslant \log{{a \choose b}} \leqslant aH\(\frac{b}{a}\).
\end{equation}
\end{lemma}

Consider the segment $RT_E$ of the curve $\Gamma$ and represent it as a monotonic function
\begin{equation}
y = f(x)
\end{equation}
supported over the interval $[V,T_E]$, see Figure \ref{fig:zoom_extr} for reference. Let the part of the boundary of the polyomino $Z$ supported by the same interval be $\kappa_n(x)$. Below we show that condition
\begin{equation}
d\(\Gamma,\partial Z\) \leqslant \varepsilon,
\end{equation}
where $\partial Z$ is the polyomino boundary curve, implies that at the points $O$ and $V$,
\begin{equation}
|y(x) - \kappa_n(x)| \leqslant \gamma(\varepsilon),\quad x \in\{O,V\},
\end{equation}
for some function $\gamma(\varepsilon) \to 0,\; \varepsilon \to 0$. Using this fact  and the last two lemmas, we can write
\begin{equation}
\log\,Q_A^{T_NN} \leqslant \log\,{(2\gamma(\varepsilon) + \tau_N)\sqrt{n} \choose 2\gamma(\varepsilon)\sqrt{n}} \leqslant \sqrt{n}(2\gamma(\varepsilon)+\tau_N)H\(\frac{\tau_N}{2\gamma(\varepsilon)+\tau_N}\),
\end{equation}
for all $n$ large enough. Here and below to keep the notation short we omit the rounding square brackets in $2\gamma(\varepsilon)\sqrt{n}$ and all similar expressions and assume the corresponding numbers to be integers. Similarly,
\begin{equation}
\label{eq:Qxx1_bound}
\log\,Q_A^{T_XX} \leqslant \sqrt{n}(2\gamma(\varepsilon)+\tau_X)H\(\frac{\tau_X}{2\gamma(\varepsilon)+\tau_X}\), X = N,E,S,W,
\end{equation}
where the index $T_XX$ is used to denote both $T_XX$ and $XT_X$ interchangeably without loss of rigor. Since $H(x)$ is bounded and $2\gamma(\varepsilon)+\tau_X \to 0,\; \varepsilon\to 0$, inequality (\ref{eq:Qxx1_bound}) immediately implies that
\begin{equation}
\label{eq:q_XX}
\lim_{\varepsilon \to 0}\frac{\log\,Q_A^{T_XX}}{\sqrt{n}} = 0,\quad X = N,E,S,W.
\end{equation}


Next we focus on bounding the value of $\log\,Q_A^{NT_E}$. The rest of the terms in the first line of (\ref{eq:z_e_calc_log}) are treated analogously.

\begin{lemma}
\label{lem:up_b_q_n}
Under the assumptions of Theorem \ref{thm:ldp_main},
\begin{equation}
\label{eq:lem_5_main}
\limsup_{n\to \infty}\frac{\log\, Q_A^{NT_E}}{\sqrt{n}} \leqslant\int_{NT_E}H\(\frac{|y'|}{1+|y'|}\)(|dx|+|dy|) + \phi^{NT_E}(\varepsilon),
\end{equation}
where $NT_E$ is the segment of $\Gamma$ in the north-eastern quadrant and $\phi^{NT_E}(\varepsilon) \to 0,\; \varepsilon \to 0$.
\end{lemma}

\begin{rem}
\label{rem:bound_rem_up}
Using exactly the same reasoning as in Lemma \ref{lem:up_b_q_n}, we can obtain similar bounds for the rest of the quadrant segments $ES, T_ST_W$ and $WT_N$ of $\Gamma$.
\end{rem}

Let us get back to the upper bound on $Q_A$. Note that $Q_A$ is a sum of $Q_A^e$-s for all possible choices of the extreme points. We know that the point $N$ can move around $T_N$ such that its abscissa belongs to the range $N_x \in [T_{N,x}-\tau_N,T_{N,x}+\tau_N]$. Similarly for the rest of the extreme points. Overall,
\begin{equation}
\label{eq:calc}
Q_A = \sum_{N \in \Delta_N, E \in \Delta_E,S \in \Delta_S,W \in \Delta_W}Q_A^e = \sum Q_A^{T_NN}Q_A^{NT_E}Q_A^{T_EE}Q_A^{ES}Q_A^{ST_S}Q_A^{T_ST_W}Q_A^{T_WW}Q_A^{WT_N},
\end{equation}
where $\Delta_N = [T_{N,x}-\tau_N,T_{N,x}+\tau_N],\; \Delta_E = [T_{E,y}-\tau_E,T_{E,y}+\tau_E],\; \Delta_S = [T_{S,x}-\tau_S,T_{S,x}+\tau_S]$, and $\Delta_W = [T_{W,y}-\tau_W,T_{W,y}+\tau_W]$. Note that the specific sequence of nodes in the superscripts of the right-hand side of (\ref{eq:calc}) is chosen according to Figure \ref{fig:zoom_extr}, and can alter for a different set $Q_A^e$, but we will always have four multipliers corresponding to the curve segments of the form $T_XX$ and four corresponding to the segments in quadrants, so it is only a matter of notation.

Using (\ref{eq:calc}), let us bound the logarithm of $Q_A$ from above,
\begin{align}
\label{eq:main_asymp_der1}
&\frac{1}{\sqrt{n}}\log\ Q_A = \frac{1}{\sqrt{n}}\log\, \(\sum_{N \in \Delta_N, E \in \Delta_E,S \in \Delta_S,W \in \Delta_W}Q_A^{T_NN}Q_A^{NT_E}Q_A^{T_EE}Q_A^{ES}Q_A^{ST_S}Q_A^{T_ST_W}Q_A^{T_WW}Q_A^{WT_N}\) \nonumber \\
& \leqslant \frac{1}{\sqrt{n}}\log\,\(16n^2\prod_X(2\tau_X)\max\limits_{N,E,S,W}Q_A^{T_NN}Q_A^{NT_E}Q_A^{T_EE}Q_A^{ES}Q_A^{ST_S}Q_A^{T_ST_W}Q_A^{T_WW}Q_A^{WT_N}\) \nonumber \\
& \leqslant \frac{1}{\sqrt{n}}\log\, \left(\max_N Q_A^{T_NN}\max_{N} Q_A^{NT_E}\max_E Q_A^{T_EE}\max_{E,S} Q_A^{ES}\max_S Q_A^{ST_S}Q_A^{T_ST_W}\max_{W} Q_A^{T_WW}\max_W Q_A^{WT_N}\right) \nonumber \\
&\qquad\qquad + \frac{\log\(16n^2\prod_X(2\tau_X)\)}{\sqrt{n}} \nonumber \\
& \stackrel{(i)}{\leqslant} \(\int\limits_{T_N}^{T_E}+\int\limits_{T_E}^{T_S}+\int\limits_{T_S}^{T_W}+\int\limits_{T_W}^{T_N}\)H\(\frac{|y'|}{1+|y'|}\)(|dx|+|dy|) + \phi(\varepsilon) + \frac{\log\(16n^2\prod_X(2\tau_X)\)}{\sqrt{n}} \nonumber \\
& = \int_\Gamma H\(\frac{|y'|}{1+|y'|}\)(|dx|+|dy|) + \phi(\varepsilon) + \frac{\log\(16n^2\prod_X(2\tau_X)\)}{\sqrt{n}},
\end{align}
where in (i) we used (\ref{eq:q_XX}), Lemma \ref{lem:up_b_q_n} and Remark \ref{rem:bound_rem_up}, and therefore $\phi(\varepsilon)\to 0$ when $\varepsilon \to 0$. Take the $\limsup_{\varepsilon\to\infty}\limsup_{n\to\infty}$ of (\ref{eq:main_asymp_der1}) to get the required bound,
\begin{equation}
\label{eq:up_b_res_m}
\limsup_{\varepsilon\to\infty}\limsup_{n\to\infty}\frac{\log\,Q_A}{\sqrt{n}} \leqslant \int_\Gamma H\(\frac{|y'|}{1+|y'|}\)(|dx|+|dy|).
\end{equation}

Let us now turn to the proof of the lower bound (\ref{eq:proof_2m}). Similarly to (\ref{eq:z_e_calc_log}), it is easy to note that
\begin{equation}
\label{eq:z_e_calc_log_2}
\log\,Q_A^e \geqslant \log\,Q_A^{NT_E} + \log\,Q_A^{T_ST_W} + \log\,Q_A^{ES} + \log\,Q_A^{WT_N}.
\end{equation}
To treat this bound we use the following result.
\begin{lemma}
\label{lem:lower_b_q_n}
Under the assumptions of Theorem \ref{thm:ldp_main},
\begin{equation}
\label{lem:lemma_6_main}
\liminf_{n\to \infty}\frac{\log\, Q_A^{NT_E}}{\sqrt{n}} \geqslant\int_{NT_E}H\(\frac{|y'|}{1+|y'|}\)(|dx|+|dy|) + \psi^{NT_E}(\varepsilon),
\end{equation}
where $NT_E$ is the segment of $\Gamma$ in the north-eastern quadrant and $\psi^{NT_E}(\varepsilon) \to 0,\; \varepsilon \to 0$.
\end{lemma}

\begin{rem}
\label{rem:bound_rem_lower}
Here again, through the same reasoning as in Lemma \ref{lem:lower_b_q_n}, we can obtain similar bounds for the segments $ES, T_ST_W$ and $WT_N$ of $\Gamma$.
\end{rem}

In a manner similar to (\ref{eq:up_b_res_m}), we get the lower bound from Lemma \ref{lem:lower_b_q_n} and Remark \ref{rem:bound_rem_lower},
\begin{equation}
\liminf_{\varepsilon\to\infty}\liminf_{n\to\infty}\frac{\log\,Q_A}{\sqrt{n}} \geqslant \int_\Gamma H\(\frac{|y'|}{1+|y'|}\)(|dx|+|dy|).
\end{equation}
Note that we can express the asymptotic behavior of $V_A$ as
\begin{equation}
\lim_{n\to \infty} \frac{\log\,V_A}{\sqrt{n}} = \max_{\Gamma \in X} I(\Gamma),
\end{equation}
and the statement of Theorem \ref{thm:ldp_main} follows.
\end{proof}


The proofs presented below follow that of Theorem 1 from \cite{blinovskii1999large} with adjustments necessary for our setup. 

\begin{proof}[Proof of Lemma \ref{lem:up_b_q_n}]
As we have already mentioned earlier, the curve segment $RT_E$ can be parametrized as a monotonically decreasing piece-wise differentiable function supported on the horizontal projection $[V,T_E]$ of $RT_E$ onto the $x$ axis,
\begin{equation}
\label{eq:alphabeta_def}
y = f(x),\quad x \in [\alpha,\beta],
\end{equation}
where we assume $f(x)$ to be positive and denote $\alpha=V,\; \beta=T_E$. For convenience, the monotonically decreasing part of the polyomino boundary $y=\kappa_n(x)$ considered here is assumed to be continuous on the right.

Partition $[\alpha, \beta]$ into $s$ closed intervals
\begin{equation}
[\alpha, \beta] = \bigcup\limits_{k=1}^{s}[c_k,q_k],
\end{equation}
intersecting only on their boundaries. Without loss of generality assume that for any $x$ which is an end point of one of the considered intervals,
\begin{equation}
\label{eq:der_cond}
y'(x)<c,
\end{equation}
for some constant\footnote{If this condition does not hold for some segment of $\Gamma$, we can always consider the other local parametrization $x=x(y)$, for which it will hold.} $c$.

Below, in the course of proving (\ref{eq:proof_1m}) we replace the requirement $\kappa_n \in U_\varepsilon(y)$, by
\begin{equation}
\label{eq:gamma_ineq}
|\kappa_n(x)-y(x)|<\gamma(\varepsilon),
\end{equation}
where $x$ runs through the end points of the intervals and $\gamma(\varepsilon)>0$. Later we explain that (\ref{eq:der_cond}) and the condition $\kappa_n \in U_\varepsilon(y)$ imply that $\gamma(\varepsilon)$ can be chosen in such a way that $\gamma(\varepsilon) \to 0,\; \varepsilon \to 0$.

For $z\leqslant 0$, define a function
\begin{equation}
\label{eq:L_def}
L(z) = (1-z)H\(\frac{-z}{1-z}\),
\end{equation}
which is continuous and satisfies
\begin{equation}
\label{eq:L_cont_b}
0 \stackrel{(i)}{\leqslant} L(z+\xi)-L(z) \stackrel{(ii)}{\leqslant} L(\xi) \to 0,\quad \xi \to 0,
\end{equation}
where (i) follows from the monotonicity of $L$ and (ii) from the relation
\begin{equation}
L'_z(z+\xi)-L'_z(z) = \log \frac{-z-\xi}{1-z-\xi}-\log\frac{-z}{1-z} \geqslant 0,\; \xi \leqslant 0.
\end{equation}
Let $n_1,\; n_2,\; \dots$ be a sequence on which $\limsup$ is reached in (\ref{eq:lem_5_main}). For any $x$ which is an end on an interval $[c_k,q_k]$ there exist at most $[2\gamma(\varepsilon)\sqrt{n_i}]$ values of $\kappa_{n_i}(x)$ for which (\ref{eq:gamma_ineq}) holds. Due to Lemma \ref{lem:count_diagr}, given the values $\kappa_{n_i}(c_k) \geqslant \kappa_{n_i}(q_k)$ of $\kappa_n$ at points $c_k$ and $q_k$ respectively, we have 
\begin{equation}
{\sqrt{n_i}\(\kappa_{n_i}(c_k) - \kappa_{n_i}(q_k) + q_k-c_k\)-1 \choose \sqrt{n_i}\(q_k-c_k\)}
\end{equation}
possibilities for the restrictions of $\kappa_n$ onto the interval $[c_k,q_k]$. Let us now bound the number $Y_{n_i}$ of the possibilities of restricting $\kappa_n$ onto the set of intervals $[c_k,q_k]$ from above as
\begin{equation}
\label{eq:Y_upper_b}
Y_{n_i} \leqslant \nprod[1.5]_{k=1}^{s}\(2\gamma(\varepsilon)\)^2n_i {\sqrt{n_i}\(\kappa_{n_i}(c_k) - \kappa_{n_i}(q_k) + q_k-c_k\)-1 \choose \sqrt{n_i}\(q_k-c_k\)}.
\end{equation}
Taking the logarithms of both sides we get,
\begin{align}
\label{eq:Y_upper_b_log}
\log\, Y_{n_i} & \leqslant s\,\log\(\(2\gamma(\varepsilon)\)^2n_i\) \\
& + \sqrt{n_i}\nsum[1.5]_{k=1}^s(q_k-c_k)\(1-\frac{\kappa_{n_i}(q_k)-\kappa_{n_i}(c_k)}{q_k-c_k}\) H\(\frac{-(\kappa_{n_i}(q_k)-\kappa_{n_i}(c_k))}{q_k-c_k-(\kappa_{n_i}(q_k)-\kappa_{n_i}(c_k)}\), \nonumber
\end{align}
where we utilized the bound (\ref{eq:stril_b}). Denote
\begin{equation}
\kappa_{n_i}(x) = y(x)+\gamma(\varepsilon,x),
\end{equation}
where
\begin{equation}
|\gamma(\varepsilon,x)| < \gamma(\varepsilon).
\end{equation}
Divide (\ref{eq:Y_upper_b_log}) by $\sqrt{n_i}$ and let $i \to \infty$. Since the entropy function $H(z)$ is convex, we can use Jensen's inequality to obtain
\begin{equation}
\label{eq:jens_ineq}
\sum_{l=1}^s(\xi_l-z_l)H\(\frac{-z_l}{\xi_l-z_l}\) \leqslant \(\sum_{l=1}^s\xi_l-\sum_{l=1}^sz_l\)H\(\frac{-\sum_{l=1}^sz_l}{\sum_{l=1}^s\xi_l-\sum_{l=1}^sz_l}\),\; \xi_l-z_l \geqslant 0.
\end{equation}
Equation (\ref{eq:L_cont_b}) implies that
\begin{equation}
\label{eq:L_bound}
L\(\frac{y(q_k)-y(c_k)+\gamma(\varepsilon,q_k)-\gamma(\varepsilon,c_k)}{q_k-c_k}\) \leqslant L\(\frac{y(q_k)-y(c_k)}{q_k-c_k}\)+ L\(\frac{\gamma(\varepsilon,q_k)-\gamma(\varepsilon,c_k)}{q_k-c_k}\).
\end{equation}
Using (\ref{eq:jens_ineq}), let us bound the contribution of the last summand of (\ref{eq:L_bound}) to (\ref{eq:Y_upper_b_log}),
\begin{equation}
\label{eq:sec_term_b}
\nsum[1.5]_{k=1}^{s}(q_k-c_k)L\(\frac{\gamma(\varepsilon,q_k)-\gamma(\varepsilon,c_k)}{q_k-c_k}\) \leqslant \nsum[1.5]_{k=1}^{s}(q_k-c_k)L\(\frac{\sum_{k}\gamma(\varepsilon,q_k)-\gamma(\varepsilon,c_k)}{\sum_{k}q_k-c_k +\gamma(\varepsilon,q_k)-\gamma(\varepsilon,c_k)}\).
\end{equation}
Since
\begin{equation}
\sum_{k}q_k-c_k = (\beta-\alpha),
\end{equation}
and
\begin{equation}
\sum_{k}\gamma(\varepsilon,q_k)-\gamma(\varepsilon,c_k) < f(\varepsilon),
\end{equation}
where $f(\varepsilon)$ can be chosen in such a way that
\begin{equation}
f(\varepsilon) \to 0,\quad \varepsilon \to 0,
\end{equation}
we conclude that the right-hand side of (\ref{eq:sec_term_b}) is of the order of
\begin{equation}
(\beta-\alpha)L\(\frac{f(\varepsilon)}{\alpha-\beta}\) \to 0,\quad\varepsilon \to 0.
\end{equation}
Therefore, the contribution of the second summand from (\ref{eq:L_bound}) to the right-hand side of (\ref{eq:Y_upper_b_log}) tends to zero together with $\varepsilon$.

Next, let us demonstrate that $\gamma(\varepsilon,x)\to 0$, for $x\in\{c_k,\;q_k\}$ when $\varepsilon \to 0$ under the condition (\ref{eq:der_cond}). Indeed, choose $x_0 \in\{c_k,\;q_k\}$. For a fixed $\omega>0$, let $h > 0$ be such small that
\begin{equation}
y(x)-y(x_0) < (c+\omega)(x-x_0),\quad 0 < x-x_0 < h.
\end{equation}
Let
\begin{equation}
\kappa_{n_i}(x_0)-y(x_0) = \gamma_1(\varepsilon) > 0.
\end{equation}
Since $\kappa_{n_i}$ is a monotonic function and
\begin{equation}
\kappa_{n_i}(x)-y(x) \geqslant 0,
\end{equation}
for $\max[0,x'] \leqslant x \leqslant x_0$, where
\begin{equation}
x' = \max\[x_0-h,\; \frac{(c+\omega)x_0-\gamma_1(\varepsilon)}{c+\omega}\],
\end{equation}
we obtain
\begin{align}
\label{eq:eps_eq}
\varepsilon &> \int_{\max[0,x']}^{x_0}|\kappa_{n_i}(x)-y(x)|dx = \int_{\max[0,x']}^{x_0}\kappa_{n_i}(x)dx - \int_{\max[0,x']}^{x_0}y(x)dx \\
&\geqslant \frac{\gamma_1(\varepsilon)}{2}\(x_0-\max[0,x']\). \nonumber
\end{align}
The last inequality basically that the integral in (\ref{eq:eps_eq}) is bounded from below by the area of the triangle determined by the lines
\begin{gather}
f_1(x) = y(x_0)+\gamma_1(\varepsilon),\quad f_2(y)=y(x_0), \\ \quad f_3(x)=y(x_0)+\frac{x_0-x}{x_0-x'}\gamma_1(\varepsilon).
\end{gather}
When $\varepsilon \to 0$, (\ref{eq:eps_eq}) implies that $\gamma_1(\varepsilon) \to 0$. Similar reasoning applies if
\begin{equation}
\kappa_{n_i}(x_0)-y(x_0) = \gamma_1(\varepsilon) < 0.
\end{equation}
Using the obtained bounds and applying $\limsup_{n_i\to \infty}$ to the both sides of (\ref{eq:Y_upper_b_log}) divided by $\sqrt{n_i}$, we get
\begin{equation}
\label{eq:new_b}
\limsup_{n_i\to \infty} \frac{\log\, Y_{n_i}}{\sqrt{n_i}} \leqslant \nsum[1.5]_{k}(q_k-c_k)L\(\frac{y(q_k)-y(c_k)}{q_k-c_k}\) + \phi(\varepsilon),
\end{equation}
where $\phi(\varepsilon) \to 0,\; \varepsilon \to 0$.

Next we increase $s$ in such a way that
\begin{equation}
w = \max_{k}(q_k-c_k) \to 0,
\end{equation}
then the first summand in the right-hand side of (\ref{eq:new_b}) becomes
\begin{equation}
\nsum[1.5]_{k=1}^{s}(q_k-c_k)L\(\frac{1}{q_k-c_k}\int_{c_k}^{q_k}y'(x)dx\) = \int_{\alpha}^{\beta}L\(y_c(x)\)dx,
\end{equation}
where $y_c(x)$ is a step function such that
\begin{equation}
y_c(x) = \int_{c_k}^{q_k}y'(x)dx,\quad x \in [c_k,q_k).
\end{equation}
Taking into consideration the last two equations and applying $\liminf_{w\to 0}$ to the both sides of (\ref{eq:new_b}), we obtain
\begin{align}
\label{eq:imp_b}
\limsup_{n\to \infty}\frac{\log\, Y_n}{\sqrt{n}} - \phi(\varepsilon) & \leqslant \liminf_{w\to 0} \int_{\alpha}^{\beta}L\(y_c(x)\)dx \\
&\stackrel{(i)}{\leqslant} \int_{\alpha}^{\beta}\liminf_{w\to 0} L\(y_c(x)\)dx \stackrel{(ii)}{=} \int_{\alpha}^{\beta}L\(\liminf_{w\to 0} y_c(x)\)dx = \int_{\alpha}^{\beta}L\(y'(x)\)dx, \nonumber
\end{align}
where (i) follows from Fatou's lemma, (ii) follows form the continuity of $L$.

Finally, from (\ref{eq:imp_b}) we have
\begin{equation}
\limsup_{n\to \infty}\frac{\log\, Y_n}{\sqrt{n}} \leqslant \int_{[\alpha,\beta]}L\(y'(x)\)dx + \phi(\varepsilon),
\end{equation}
and (\ref{eq:lem_5_main}) follows.
\end{proof}

\begin{proof}[Proof of Lemma \ref{lem:lower_b_q_n}]
Consider a subsequence $n_i$ on which the $\liminf$ is attained in (\ref{lem:lemma_6_main}). Define the interval $[\alpha,\beta]$ exactly as in (\ref{eq:alphabeta_def}), partition it into $s$ equal intervals $[a_j,b_j],\;j=1,\dots,s$, and denote their lengths by
\begin{equation}
\Delta = b_j-a_j = \frac{\beta-\alpha}{s}.
\end{equation}
Above we focused on the upper bound and considered an excessive number of functions $\kappa_n$. Indeed, some of $\kappa_n$ did not belong to $U_{\varepsilon}(y)$, moreover, some of them could not serve as boundaries of the polyominoes under consideration because since their areas could be larger than the area $A^{NT_E}$ of the quadrant of $G$ at hand. Now we treat the lower bound and must only count those $\kappa_n$ that are the boundaries of convex polyominoes of area $A^{NT_E}$ belonging to $U_{\varepsilon}(y)$.

Consider those $\kappa_n$ which for every $x_0 \in \{a_j,b_j\}$ take the same value $\kappa_n(x_0)$ and satisfy the condition
\begin{equation}
\label{eq:cond_25}
|\kappa_{n}(x_0)-y(x_0)| \leqslant \frac{1}{\sqrt{n}}.
\end{equation}
Assume we build our diagram from left to right. Two issues can happen during the course of such construction under the condition (\ref{eq:cond_25}): 
\begin{enumerate}[1)]
\item we can exhaust the area $A^{NT_E}$ before we reach the rightmost point of $y$,
\item we can reach the rightmost point of $y$ having diagram of a smaller area than required.
\end{enumerate}
Later we show that in the case 1) the remaining area is small and can be spread above the constructed diagram without pushing it beyond $U_{\varepsilon}(y)$, and in the case 2) the total length of the remaining not covered intervals $[a_j,b_j]$ can be made arbitrarily small. Roughly speaking, we need to show that the areas under the curves $\kappa_n(x)$ and $y(x)$ for $x \in [0,\eta],\; \eta < \beta$ are close, where $\eta$ is the point where $\kappa_n$ becomes zero for the first time. 
During the construction of a diagram from left to right as described above, if we reached $\beta$, then we spread the remaining cells above the already constructed diagram. Let us show that the total area of these extra cells can be made arbitrarily small. Indeed, since $\kappa_n \in U_{\varepsilon}(y)$ the area excess can be expressed as
\begin{equation}
\label{eq:cond_27_a}
A^{NT_E} - \int_\alpha^\beta \kappa_n(x) dx = \int_\alpha^\beta y(x) - \kappa_n(x) dx \leqslant \xi,
\end{equation}
where $\xi\to 0$ when $\Delta$ shrinks and $n$ grows. 

Next we recycle the ideas used for the proof of (\ref{eq:eps_eq}), but this time we will also upper bound the $L^1$-distance between the curves. For two monotonically non-increasing functions $z_1(x)$ and $z_2(x)$, such that $|z_1(x)-z_2(x)|\leqslant 1/\sqrt{n}$ for $x=a,\;b$ where $a<b$ are arbitrary reals, we clearly have
\begin{equation}
\label{eq:ineq_28}
\int_a^b|z_1(x)-z_2(x)|dx \leqslant (b-a)\(z_1(x)-z_2(x)+\frac{2}{\sqrt{n}}\).
\end{equation}

Assume that (\ref{eq:cond_25}) holds for all $x\in\{a_j,b_j\}$, then from (\ref{eq:ineq_28}) we get
\begin{align}
\label{eq:eq_delta_29}
\int_{[\alpha,\beta]} |\kappa_n(x)-y(x)| dx & = \nsum[1.5]_{j=1}^s \int_{[a_j,b_j]} |\kappa_n(x)-y(x)| dx \\ 
&\leqslant \nsum[1.5]_{j=1}^s (b_j-a_j)\(y(a_j)-y(b_j)+\frac{2}{\sqrt{n}}\) \leqslant \Delta\(y(\alpha)-y(\beta)+\frac{2}{\sqrt{n}}\). \nonumber
\end{align}

Let now $\eta < \beta$ so that $\kappa_n(x)=0,\; |y(x)-\kappa_n(x)|>\frac{1}{\sqrt{n}}$ for $x>\eta$ and $\kappa_n(x)>0$ for $x<\eta$. This implies that
\begin{equation}
\int_0^{\eta}\kappa_n(x)dx = A^{NT_E},
\end{equation}
\begin{equation}
\int_0^{\eta}y(x)dx = A^{NT_E}-\rho < \int_0^\beta y(x)dx,
\end{equation}
Now
\begin{equation}
\int_0^{\eta}|\kappa_n(x)-y(x)| dx > \left|\int_0^{\eta}\kappa_n(x)-y(x) dx\right| > \rho.
\end{equation}
On the other hand, the left-hand side of the last inequality is bounded from above by the expression in the right-hand side of (\ref{eq:eq_delta_29}). As a consequence, for small enough $\Delta$, the value of $\rho$ must be also small,
\begin{equation}
\rho \to 0,\quad \Delta \to 0.
\end{equation}
This is only possible if $\eta$ is large enough. Let
\begin{equation}
\theta = \beta - \eta,
\end{equation}
then
\begin{equation}
\int_{\eta}^{\beta}y(x)dx \to 0,\quad \Delta \to 0.
\end{equation}
As we have already mentioned above, for $\Delta$ small enough, condition (\ref{eq:cond_27_a}) will hold with $\xi$ small. Overall, (\ref{eq:cond_27_a}) and (\ref{eq:eq_delta_29}) imply that the constructed $\kappa_n(x)$ will belong to $U_{\varepsilon}(y)$ under the appropriate choice of $\Delta$.

Let $Y_{n_i}$ be the number of admissible diagrams we count, then it is lower bounded by the number of $\kappa_{n_i}$ satisfying (\ref{eq:cond_27_a}) and (\ref{eq:eq_delta_29}). By Lemma \ref{lem:count_diagr}, the number of possible restrictions of $\kappa_{n_i}(x)$ onto $[a_j,b_j]$ for given $\kappa_{n_i}(a_j)$ and $\kappa_{n_i}(b_j)$ is
\begin{equation}
{\sqrt{n_i}\(\kappa_{n_i}(a_j) - \kappa_{n_i}(b_j) + b_j-a_j\)-1 \choose \sqrt{n_i}\(b_j-a_j\)}.
\end{equation}
The number of such restrictions on $[\alpha,\beta]$, for the given $\kappa_{n_i}(a_j)$ and $\kappa_{n_i}(b_j)$, is lower bounded by the product
\begin{equation}
\nprod[1.5]_{j=1}^{l} {\sqrt{n_i}\(\kappa_{n_i}(a_j) - \kappa_{n_i}(b_j) + b_j-a_j\)-1 \choose \sqrt{n_i}\(b_j-a_j\)},
\end{equation}
where $l$ is calculated as follows. Let $\eta < \beta$ and $r$ be the largest number such that $\mu\([a_r,b_r] \cap [\eta,\beta]\)=0$, then set $l=s$. Otherwise set $l = r+1$ and $a_{l}=a_{r+1},\; b_{l}=\eta$. Clearly,
\begin{equation}
\mu\(\bigcup_{j=l}^s[a_j,b_j]\) < \beta - \eta + b_{l} - a_{l} < \theta + \Delta \to 0,\quad \Delta \to 0.
\end{equation}
Using the bound
\begin{equation}
\log {{m \choose s}} \geqslant mH\(\frac{s}{m}\)+o(m),\quad m \to \infty,
\end{equation}
following from (\ref{eq:stril_b}) and taking into account that
\begin{equation}
|\kappa_{n_1}(x)-y(x_0)| \leqslant \frac{1}{\sqrt{n_i}},\quad x=a_j,\; b_j,\; j\leqslant l,
\end{equation}
we get
\begin{equation}
\frac{\log\, Y_{n_i}}{\sqrt{n_i}} \geqslant \nsum[1.5]_{j=1}^{l}\(1-\frac{y(b_j)-y(a_j)}{b_j-a_j}+O\(\frac{1}{\sqrt{n_i}}\)\)H\(\frac{-\frac{y(b_j)-y(a_j)}{b_j-a_j}+O\(\frac{1}{\sqrt{n_i}}\)}{1-\frac{y(b_j)-y(a_j)}{b_j-a_j}+O\(\frac{1}{\sqrt{n_i}}\)}\).
\end{equation}
Let $n_i \to \infty$ and recall the definition of $L$ from (\ref{eq:L_def}) to obtain,
\begin{align}
\liminf_{n\to\infty}\frac{\log\, Y_{n_i}}{\sqrt{n_i}} &\geqslant \nsum[1.5]_{j=1}^{l}(b_j-a_j)L\(\frac{y(b_j)-y(a_j)}{b_j-a_j}\) = \nsum[1.5]_{j=1}^{l}(b_j-a_j)L\(\frac{1}{b_j-a_j}\int\limits_{a_j}^{b_j}y'(x)dx\) \nonumber \\
& \stackrel{(i)}{\geqslant} \nsum[1.5]_{j=1}^{l}\int\limits_{a_j}^{b_j}L\(y'(x)\)dx = \int\limits_{\alpha}^{\min[\eta,\beta]}L\(y'(x)\)dx,
\end{align}
where in (i) we utilized the convexity of $L$ together with Jensen's inequality. Now let $\Delta \to 0$ to obtain
\begin{equation}
\liminf_{n \to \infty} \frac{\log\, Y_n}{\sqrt{n}} \geqslant \int\limits_{T_NT_E}L\(y'(x)\)dx + \psi(\varepsilon),
\end{equation}
which completes the proof.
\end{proof}

\begin{rem}
Assume now that instead of fixed area we deal with convex polyominoes of fixed perimeter. This case is even simpler since for most of the polyominoes the perimeter constraint will never be active. Indeed, the perimeter constraint only plays role if the diagram $\kappa_n(x)$ hits the boundary of the circumscribing rectangle. By appropriate choice of the extreme points of the polyomino we can easily satisfy this requirement, thus the bulk of the diagram will not be affected by it. The same applies to the polyominoes with both perimeter and area fixed (unless $\int_{\Gamma}(1+|\tan \Gamma|)H\(\frac{1}{1+|\cot \Gamma|}\)dx = 0$, which is not the case we consider).
\end{rem}

\bibliographystyle{IEEEtran}
\bibliography{ilya_bib}
\end{document}